\documentclass[10pt]{amsart}
\usepackage[utf8]{inputenc}
\usepackage{amssymb}
\usepackage{amscd}
\usepackage{amsfonts}
\usepackage{amsmath}
\usepackage{amsthm}
\usepackage[all]{xy}
\usepackage{mathrsfs}
\usepackage[noadjust]{cite}
\usepackage[colorlinks=true, allcolors=blue]{hyperref}
\usepackage[shortlabels]{enumitem}
\usepackage{tikz}

\usepackage{bm}
\usepackage{comment}

\usepackage[mathscr]{eucal}

\newtheorem{theorem}{Theorem}[section]
\newtheorem*{theorem*}{Theorem}
\newtheorem{lemma}[theorem]{Lemma}
\newtheorem{proposition}[theorem]{Proposition}
\newtheorem{corollary}[theorem]{Corollary}

\theoremstyle{definition}
\newtheorem{definition}[theorem]{Definition}
\newtheorem{example}[theorem]{Example}

\newtheorem{remark}[theorem]{Remark}

\newcommand{\pcoor}[1]{%
  \begingroup\lccode`~=`: \lowercase{\endgroup
  \edef~}{\mathbin{\mathchar\the\mathcode`:}\nobreak}%
  [
  \begingroup
  \mathcode`:=\string"8000
  #1%
  \endgroup 
  ]
} 

\newcommand{\ltens}{\mathbin{\displaystyle\stackrel{L}{\otimes}}} 

\newcommand{\leqp}{%
  \mathrel{\raisebox{-0.5ex}{$\scriptscriptstyle($}}%
  \leq
  \mathrel{\raisebox{-0.5ex}{$\scriptscriptstyle)$}}%
}

\newenvironment{psmallmatrix}
  {\left(\begin{smallmatrix}}
  {\end{smallmatrix}\right)}

\newcommand{\cat}[1]{\mathsf{#1}}
\newcommand{\st}[1]{\mathfrak{#1}}

\newcommand{\sh}[1]{\mathcal{#1}}

\newcommand{\tail}[1]{\mathbf{#1}}
\newcommand{\cpx}[1]{\mathbb{#1}}
\newcommand{\chch}[1]{\mathbf{#1}}

\newcommand{\grsh}[1]{\mathcal{#1}}

\def\Sch{\cat{Sch}}
\def\gr{\cat{gr}}
\def\qgr{\cat{qgr}}
\def\tor{\cat{tor}}

\def\Spec{\mathrm{Spec}}
\def\Hom{\mathrm{Hom}}

\def\vir{\mathrm{vir}}
\def\Hilb{\mathrm{Hilb}}
\def\Proj{\mathrm{Proj}}
\def\Coh{\mathrm{Coh}}

\def\DT{\mathrm{DT}}

\def\Ext{\mathrm{Ext}}

\def\Quot{\mathrm{Quot}}

\def\SHom{\sh{H}\mathrm{om}}

\def\op{\mathrm{op}}
\def\pd{\mathrm{pd}}
\def\der{\mathcal{D}}
\def\id{\mathrm{id}}
\def\JH{\mathrm{JH}}

\def\ltensor{\ltens}

\def\at{\mathrm{at}}
\def\ob{\mathrm{ob}}

\title{Donaldson--Thomas theory of quantum Fermat quintic threefolds I}

\author{Yu-Hsiang Liu}
\address{Department of Mathematics, University of British Columbia}
\email{\href{mailto:yuliu@math.ubc.ca}{yuliu@math.ubc.ca}}
\thanks{} 

\begin{document}

\begin{abstract} 
In this paper, we study non-commutative projective schemes whose associated non-commutative graded algebras are finite over their centers. We study their moduli spaces of stable sheaves, and construct a symmetric obstruction theory in the Calabi--Yau-3 case. This allows us to define Donaldson--Thomas type invariants. We also discuss the simplest examples, called quantum Fermat quintic threefolds.
\end{abstract}

\maketitle

\setcounter{tocdepth}{1}

\tableofcontents

\section{Introduction}

In \cite{Tho00}, Thomas introduced the integer-valued invariants, now called Donaldson--Thomas invariants, for smooth Calabi--Yau threefolds by integrations over the degree zero virtual fundamental classes for moduli spaces of stable sheaves. Later it was discovered by Behrend (\cite{Beh09}) that these invariants can be computed by the weighted Euler characteristics of certain constructible functions on the moduli spaces. Since then, there have been many generalizations, including Joyce--Song's generalized Donaldson--Thomas theory (\cite{JS12}) and Kontsevich--Soibelman's motivic (cohomological) Donaldson--Thomas theory (\cite{KS08}), which conjecturally works for any Calabi--Yau-3 (triangulated) category.

Our motivation starts from a special Calabi--Yau-$3$ category, called the quantum Fermat quintic threefold, constructed in \cite{Kan15}. A quantum Fermat quintic threefold $\qgr(A_Q)$ is a non-commutative projective scheme (\cite{AZ94}) defined by a graded algebra
\[
A=\mathbb{C}\langle t_0,\ldots,t_4\rangle\Big/\left(\sum_{n=0}^4 t_n^5,t_it_j-q_{ij}t_jt_i \right),
\]
for certain quantum parameters $q_{ij}$'s. This category is not equivalent to the category of coherent sheaves on any Calabi--Yau threefold. On the other hands, most known examples in the study of non-commutative Donaldson--Thomas theory (\cite{Sze08}, \cite{CMPS}) are given by quivers with potentials, which are non-commutative analogues of local Calabi--Yau threefolds. The quantum Fermat quintic threefold is projective, in the sense that its moduli spaces are expected to be projective. Therefore we may define Donaldson--Thomas type invariants via integration over the virtual fundamental class on the moduli spaces, which is essential for deformation invariance.

The purpose of this paper is to develop the Donaldson--Thomas theory on quantum Fermat quintic threefolds. The first step is to construct moduli spaces. We observe that $A$ contains a commutative subalgebra $B=\mathbb{C}[x_0,\ldots,x_4]/(\sum_{n=0}^4 x_n)\subset Z(A)$, where $x_i=t_i^5$, such that $A$ is a finite (graded) $B$-module. Take $X=\Proj(B)\cong\mathbb{P}^3$, then $A$ naturally induces a sheaf $\sh{A}$ of non-commutative $\sh{O}_X$-algebras on $X$. It can be shown that there is an equivalence of categories between $\qgr(A)$ and the category $\Coh(\sh{A})$ of coherent $\sh{A}$-modules (Proposition~\ref{prop:qgr-to-coh}).

For a general smooth projective variety $X$ with a coherent sheaf $\sh{A}$ of non-commutative $\sh{O}_X$-algebras, the stability condition for $\sh{A}$-modules and their moduli spaces have been studied by Simpson \cite{Sim94}. Fix a polarization $\sh{O}_X(1)$ on $X$, then the stability condition for $\sh{A}$-modules is defined via the Hilbert polynomials on $X$ similarly to Gieseker stability for coherent sheaves. Simpson's moduli spaces of (semi)stable $\sh{A}$-modules have all the desired properties.


The next step is to construct a virtual fundamental class for the moduli space $M$ of stable $\sh{A}$-modules by showing $M$ admits a perfect obstruction theory (\cite{BF97}). We start by constructing an obstruction theory on the moduli space $M$.

\begin{theorem}[= Corollary~\ref{cor:ob_for_M}]
Let $\sh{F}$ be the universal family of twisted $\sh{A}$-modules on $X\times M$. Then there is an obstruction theory
\[
\cpx{E}:=\Big(R\pi_{M*}R\SHom_{\sh{A}_M}(\sh{F},\sh{F})\Big)^{\vee}[-1]\to\cpx{L}_M
\]
for $M$.
\end{theorem}

It is well-known that deformation-obstruction theory for objects in any abelian category are governed by their $\Ext^1$ and $\Ext^2$ groups. However, as pointed out in \cite{HT10}, a more explicit description for the obstruction classes is required to construct an obstruction theory for the moduli space.

\begin{theorem}[= Theorem~\ref{thm:at-cl}]
Let $S$ be a scheme, and $\sh{F}$ be a $\sh{O}_S$-flat coherent $\sh{A}_S$-module on $X\times S$. There exists a natural class
\[
\at_{\sh{A}}(\sh{F})\in\Ext^1_{\sh{A}_S}(\sh{F},\sh{F}\otimes\pi^*_S\cpx{L}_S),
\]
called the Atiyah class, such that for any square-zero extension $S\subset\overline{S}$ with ideal sheaf $I$, an $\sh{A}$-module extension of $\sh{F}$ over $\overline{S}$ exists if and only if the obstruction class
\[
\ob=\Big(\sh{F}\xrightarrow{\at_{\sh{A}}(\sh{F})}\sh{F}\otimes\pi_S^*\cpx{L}_S[1]\xrightarrow{\id_{\sh{F}}\otimes\pi_S^*\kappa(S/\overline{S})[1]}\sh{F}\otimes \pi_S^* I[2]\Big)\in\Ext^2_{\sh{A}_S}(\sh{F},\sh{F}\otimes\pi_S^* I)
\]
vanishes, where $\kappa(S/\overline{S})\in\Ext^1_S(S,I)$ is the Kodiara--Spencer class for the extension $S\subset\overline{S}$. Moreover, if an extension of $\sh{F}$ over $\overline{S}$ exists, then all (equivalence classes of) extensions form an affine space over $\Ext^1_{\sh{A}}(\sh{F},\sh{F}\otimes\pi_S^*I)$.
\end{theorem}

Assume that $\sh{A}$ is Calabi--Yau-$3$ in the sense that $\Coh(\sh{A})$ is a Calabi--Yau-3 category. For example, the quantum Fermat quintic $\qgr(A)$ will give such $X$ and $\sh{A}$. In this case, we construct a symmetric bilinear form
\[
\theta:(R\pi_{M*}\cpx{E})^{\vee}[-1]\to R\pi_{M*}\cpx{E}\,[2]=\Big((R\pi_{M*}\cpx{E})^{\vee}[-1]\Big)^{\vee}[1].
\]
This leads to the following:

\begin{theorem}[= Theorem~\ref{thm:sym_ob}]
The truncation of the obstruction theory $\cpx{E}\to\cpx{L}_M$ for $M$ induces a symmetric obstruction theory
\[
\Big(\tau^{[1,2]}R\pi_{M*}R\SHom_{\sh{A}_M}(\sh{F},\sh{F})\Big)^{\vee}[-1]\to\cpx{L}_M,
\]
which in particular is a perfect obstruction theory.
\end{theorem}

Therefore the moduli space $M$ carries a degree zero virtual fundamental class $[M]_{\vir}\in A_0(X)$. If $M$ is projective (i.e. there is no strictly semistable $\sh{A}$-modules), then this allows us to define the Donaldson--Thomas type invariant via integration over $[M]_{\vir}$
\[
\DT(M):=\int_{[M]^{\vir}} 1,
\]
which is equal to the weighted Euler characteristic $\chi(M,\nu_M)$, as in \cite{Beh09}.

Next, we recall that classical Donaldson--Thomas theory is a virtual count of curves on a (Calabi--Yau) 3-fold. This is done by realizing the Hilbert schemes of curves are isomorphic to (components) of moduli spaces of stable sheaves \emph{with fixed determinant}. For a general sheaf $\sh{A}$ of non-commutative algebras, we do not have the notion of determinant or trace, but we can show a similar result holds under suitable assumptions.

\begin{theorem}[= Proposition~\ref{prop:dt-on-hilb}]
Assume that the stalk $\sh{A}_{\eta}$ at generic point $\eta\in X$ is simple, and $H^1(X,\sh{A})=0$. Then for any polynomial $h$ with degree $\leq 1$, there exists an open immersion
\[
\Hilb^h(X,\sh{A})\to M^{s,h_0-h}(X,\sh{A}),
\]
where $h_0$ is the Hilbert polynomial of $\sh{A}$.
\end{theorem}

In general, $M=M^{s,h_0-h}(X,\sh{A})$ may not be projective, but the Hilbert scheme $\Hilb^h(X,\sh{A})$ is always projective. Then the symmetric obstruction theory of $M$ pulls back to a symmetric obstruction theory on $\Hilb^p(X,\sh{A})$, which allows us to define Donaldson--Thomas invariants via integrations
\[
\DT^h(X,\sh{A}):=\int_{[\Hilb^{h}(X,\sh{A})]_{\vir}}1=\chi(\Hilb^{h}(X,\sh{A}),\nu).
\]

\subsection*{Acknowledgements} I would like to thank my advisor Kai Behrend for introducing me to this topic, and for his consistent guidance and support. I would also like to thank Toni Annala, Jim Bryan, Nina Morishige and Stephen Pietromonaco for many helpful conversations.

\subsection*{Notation} All schemes or algebras are separated and noetherian over $\mathbb{C}$. All (sheaves of) algebras are associative and unital. By non-commutative, we mean not necessarily commutative, and we assume that non-commutative rings are both left and right noetherian. For a non-commutative ring $A$, a $A$-module is always a left $A$-module. We will use $A^{\op}$-module to denote right $A$-module. All rings without specified non-commutative are commutative.

\section{Quantum Fermat quintic threefolds}\label{sec:ncps-to-ncv}

\subsection{Non-commutative projective schemes}

We first review the notion of non-commutative projective schemes defined by Artin and Zhang (\cite{AZ94}).

Let $A$ be a locally finite $\mathbb{Z}_{\geq 0}$-graded $\mathbb{C}$-algebra, by which we mean $A=\oplus_{i\geq 0}A_i$ and each $A_i$ is finite-dimensional. We define
\[
\qgr(A)=\gr(A)/\tor(A)
\]
to be the quotient abelian category, where $\gr(A)$ is the category of finitely-generated graded $A$-modules, and $\tor(A)$ is its Serre subcategory consisting of torsion modules, here a graded $A$-module $M$ is said to be torsion if each element $m\in M$ is annihilated by $A_{\geq n}$ for some $n$.

\begin{definition}[\cite{AZ94}]
The \emph{non-commutative projective scheme} defined by $A$ is a triple
\[
(\qgr(A),\tail{A},[1]),
\]
where $\tail{A}$ is the object in $\qgr(A)$ corresponding to $A$ as an $A$-module, and $[1]$ is the functor induced by sending any graded module $M$ to $M[1]$ defined by $(M[1])_i=M_{i+1}$.
\end{definition}

More generally, we consider a triple $(\cat{C},\tail{O},s)$ of an abelian category $\cat{C}$, an object $\tail{O}$, and a natural equivalence $s:\cat{C}\to\cat{C}$. A morphism $(\cat{C}_1,\tail{O}_1,s_1)\to(\cat{C}_2,\tail{O}_2,s_2)$ between two such triples consists of a functor $F:\cat{C}_1\to\cat{C}_2$, an isomorphism $F(\tail{O}_1)\cong\tail{O}_2$, and a natural isomorphism $s_2\circ F\cong F\circ s_1$. This morphism is an isomorphism if $F$ is an equivalence of categories.

\begin{definition}[\cite{AZ94}]
A non-commutative projective scheme is a triple $(\cat{C},\tail{O},s)$ which is isomorphic to $(\qgr(A),\tail{A},[1])$ for some graded algebra $A$.
\end{definition}

For example, if $A$ is commutative and generated by degree $1$ elements, then by Serre's theorem, $(\qgr(A),\tail{A},[1])$ is isomorphic to the triple $(X,\sh{O}_X,-\otimes\sh{O}_X(1))$, where $X=\Proj(A)$.

More generally, let $X$ be a (smooth) projective variety with a polarization $\sh{O}_X(1)$, and $\sh{A}$ a coherent sheaf of non-commutative $\sh{O}_X$-algebras on $X$. We may consider the \emph{homogeneous coordinate ring} of $(X,\sh{A})$
\[
A=\bigoplus_{n=0}^{\infty} H^0(X,\sh{A}(n)),
\]
which is naturally a graded algebra via the multiplication $\sh{A}\otimes\sh{A}\to\sh{A}$.

\begin{proposition}
The triple $(\Coh(\sh{A}),\sh{A},-\otimes\sh{O}_X(1))$ is isomorphic to the non-commutative projective scheme defined by the graded algebra $A$. In particular, there is an equivalence of (abelian) categories $\Coh(\sh{A})\cong\qgr(A)$.
\end{proposition}

\begin{proof}
This follows directly from \cite[Theorem 4.5]{AZ94}.
\end{proof}

If the graded algebra $A$ is given by such $(X,\sh{A})$, then the homogeneous coordinate ring $B:=\oplus_n H^0(X,\sh{O}_X(n))$ of $X$ is a graded subalgebra of $A$, which is contained in the center $Z(A)$. Since $\sh{A}$ is coherent, $\sh{A}(n)$ is generated by global sections for sufficiently large $n$. This implies that $A$ is a finite $B$-module.

\subsection{Algebras finite over their centers}

Let $A$ be a non-commutative graded algebra. For simplicity, we assume both $A$ and $Z(A)$ are finitely-generated. Suppose there exists a graded subalgebra $B\subset Z(A)$ of $A$ such that $A$ is a finite $B$-module. We consider $X=\Proj(B)$ and
\[
\sh{A}=\widetilde{{}_B A}
\]
the coherent sheaf on $X$ corresponding to the graded $B$-module ${}_B A$. Then $\sh{A}$ is naturally a sheaf of non-commutative $\sh{O}_X$-algebras. 

From now on we assume that $A$ is of finite global dimension. In fact, this forces $A$ to be an Artin--Schelter regular algebra since $A$ is finite over its center $Z(A)$. We omit the details and only mention two important consequences that $A$ satisfies the technical $\chi$ condition in the study of non-commutative projective schemes, and the category $\qgr(A)$ has finite cohomological dimension (which equals to the global dimension of $A$ minus $1$).

\begin{proposition}\label{prop:qgr-to-coh}
Suppose $A$ is generated by degree one elements. Then there is an equivalence of (abelian) categories $\qgr(A)\cong\Coh(\sh{A})$.
\end{proposition}

\begin{proof}
If the chosen $B$ is also generated by degree one elements, then we have
\[
\Hom_{\qgr(A)}({}_A A,M)=\Hom_{\qgr(B)}({}_B B,{}_B M)=H^0(X,\widetilde{{}_B M})
\]
for any graded $A$-module $M$. Thus \cite{AZ94}, Theorem 4.5(2) states that there is a morphism
\[
A\to A':=\bigoplus_{n=0}^{\infty}\Hom_{\qgr(A)}({}_A A,{}_A A[n])=\bigoplus_{n=0}^{\infty} H^0(X,\sh{A}(n))
\]
of graded algebras which is an isomorphism at sufficiently large degrees. This then implies that $\qgr(A)\cong\qgr(A')\cong\Coh(\sh{A})$.

In general, we choose $k$ such that the graded algebra $B^{(k)}:=\oplus_i B_{ki}$ is generated by degree one elements. Then it is well-known that $X=\Proj(B)=\Proj(B^{(k)})$, and ${}_B A$ and ${}_{B^{(k)}} A^{(k)}$ define the same coherent sheaf on $X$. In fact, the same result also holds non-commutative projective schemes (\cite[Proposition 5.10]{AZ94}). Therefore we have equivalences of categories
\[
\qgr(A)\cong\qgr(A^{(k)})\cong\Coh(\sh{A}).
\]
\end{proof}

\begin{remark}
In \cite{AZ01}, Artin and Zhang define the notion of (flat) families of objects in any abelian category. One may check that the equivalence $\qgr(A)\cong\Coh(\sh{A})$ in Proposition~\ref{prop:qgr-to-coh} induces equivalence between families, which also preserves flatness. This implies that the Hilbert schemes $\Hilb^h(A)$ constructed in \cite{AZ01} agrees with Simpson's Hilbert schemes $\Hilb^h(X,\sh{A})$ (see \ref{prop:hilb-A}). In particular, this proves the projectivity of $\Hilb^h(A)$ under the assumption of Proposition~\ref{prop:qgr-to-coh}.
\end{remark}

\begin{remark}
If $A$ is not generated by degree one elements, one may take the \emph{stacky} $\Proj$
\[
\st{X}=\mathfrak{Proj}(B)=\big[(\Spec(B)\setminus\{0\})/\mathbb{C}^*\big],
\]
where $\mathbb{C}^*$ acts on $\Spec(B)$ via the grading. Then $\st{X}$ is a Deligne--Mumford stack with a projective coarse moduli scheme $X=\Proj(B)$. The graded algebra $A$ also defines a sheaf $\sh{A}$ of $\sh{O}_{\st{X}}$-algebras on $\st{X}$. Then there is also an equivalence of categories $\Coh(\sh{A})\cong\qgr(A)$ by the same argument.
\end{remark}

Finally, since $Z(A)$ is a finitely-generated commutative algebra, by Noether normalization lemma, there exists a regular subalgebra $B\subset Z(A)$ such that $Z(A)$ is finite over $B$, hence $A$ is finite over $B$. Thus we can always choose $B$ so that $X=\Proj(B)$ is smooth (in fact, a projective space).

\begin{example}
Consider quantum projective spaces which are non-commutative projective schemes defined by quantum polynomial rings
\[
A=\mathbb{C}\langle x_0,\ldots, x_n\rangle\big/\left( x_ix_j-q_{ij}x_jx_i\right),
\]
where $q_{ii}=q_{ij}q_{ji}=1$ for all $i,j$. If $q_{ij}$'s are roots of unity, then $A$ is finite over its center $Z(A)$.
\end{example}

\subsection{Quantum Fermat quintic threefolds}\label{sec:QFQ}

In this section, we study quantum Fermat quintic threefolds introduced in \cite{Kan15}.

\begin{definition}[\cite{Kan15}]
A \emph{quantum Fermat quintic threefold} is a non-commutative projective scheme associated to a graded algebra
\[
A=\mathbb{C}\langle t_0,\ldots,t_4\rangle\big/\left( \sum_{k=0}^4 t_k^5, t_it_j-q_{ij}t_jt_i\right),
\]
with $\deg(t_i)=1$ for all $i$ and $q_{ij}$'s satisfying
\begin{enumerate}[(i)]
    \item $q_{ij}$ is a $5$-th root of unity.
    \item $q_{ii}=q_{ij}q_{ij}=1$ for all $i,j$.
    \item $\prod_j q_{ij}$ is independent of $i$.
\end{enumerate}
\end{definition}

Geometrically, a quantum Fermat quintic threefold is the Fermat quintic lying in a quantum projective $4$-space. We require condition (i) so that the quintic equation is a central element. The condition (iii) is equivalent to the category $\qgr(A)$ being CY3 \cite[Theorem 2.1]{Kan15}. 
\vspace{1mm}

Before we proceed, we recall Zhang's twisted graded algebras \cite{Zha96}.

Let $A=\oplus_i A_i$ be a graded algebra. For any automorphism $\sigma:A\to A$ of graded algebras, Zhang defines a new multiplication on $A$ by
\[
x\ast y = x\cdot \sigma^m(y)
\]
for any $x\in A_m$ and $y\in A_n$. The graded algebra with the new multiplication $\ast$ is called a twisted graded algebra of $A$, denote by $A^{\sigma}$. Then it is shown that non-commutative projective schemes associated to $A$ and any twisted algebra $A^{\sigma}$ are isomorphic.

\begin{definition}[\cite{Kan15}]
A quantum Fermat quintic threefold is called \emph{generic} if for all distinct $i,j$ and $k$,
\[
q_{ij}q_{jk}\neq q_{ik}.
\]
\end{definition}

This condition gives the ``maximal non-commutativity'' in the following sense. If $q_{ij}q_{jk}=q_{ik}$ for some distinct $i,j,k$, then we take the automorphism $\sigma:A\to A$ defined by $x_j\mapsto q_{ij}x_j$ and $x_k\mapsto q_{ik}x_k$. The twisted algebra $A^{\sigma}$ has quantum parameters $q_{ij}=q_{jk}=q_{ik}=1$, which means the non-commutative projective scheme $\qgr(A)$ contains a commutative closed subscheme of positive dimension.

\vspace{1mm}

We fix a primitive $5$-th root $q$ of unity. Any quantum Fermat quintic threefold is determined by a skew-symmetric matrix $N=(n_{ij})_{i,j}\in M_5(\mathbb{Z}/5\mathbb{Z})$ such that $(1,1,1,1,1)^{\intercal}$ is an eigenvector of $N$. We denote by $A_N$ the graded algebra with quantum parameters $q_{ij}=q^{n_{ij}}$.

\begin{theorem}
Up to a possible change of the primitive root $q\in\mu_5$, all generic quantum Fermat quintic threefold are isomorphic as non-commutative projective schemes.
\end{theorem}

\begin{proof}
Consider the following actions on the set of above matrices $N$'s:
\begin{enumerate}
    \item \textbf{Change of the primitive root $q$.} For any element $a\in(\mathbb{Z}/5\mathbb{Z})^{\times}$, changing $q$ to $q^a$ is equivalent to multiple all elements $n_{ij}$ in $N$ by $a$.
    \item \textbf{Change of variables $t_i$'s.} For any permutation $\sigma\in S_5$, we consider the change of variables $\tilde{t}_i=t_{\sigma(i)}$. Then the new graded algebra has quantum parameters given by $\tilde{n}_{ij}=n_{\sigma(i),\sigma(j)}$.
    \item \textbf{Twisted graded algebras.} For any $(a_0,\ldots,a_4)\in(\mathbb{Z}/5\mathbb{Z})^5$, let $\sigma:A_N\to A_N$ be the automorphism defined by $\sigma(x_i)=q^{a_i}x_i$. Then the twisted algebra $(A_N)^{\sigma}$ has quantum parameters given by
    \[
    N^{\sigma}:=(n_{ij}+a_i-a_j)_{i,j}.
    \]
\end{enumerate}
The proof is done with the aid of computer. There are precisely $3000$ choices of $N$'s for generic quantum Fermat quintic threefolds. All of them are equivalent under the three actions above.
\end{proof}

Due to this result, we will call it \emph{the} generic quantum Fermat quintic threefold. It will be the central subject of our study in the sequel.

\section{Sheaves of non-commutative algebras}\label{sec:nc-sheaf-alg}

Motivated by the previous section, we consider a smooth projective variety $X$ with a coherent sheaf $\sh{A}$ of non-commutative $\sh{O}_X$-algebras on $X$. We view $(X,\sh{A})$ as a ringed space, and $\pi:(X,\sh{A})\to (X,\sh{O}_X)$ is the morphism of ringed spaces given by the unit map $\sh{O}_X\to\sh{A}$. Let $\Coh(\sh{A})$ be the category of coherent $\sh{A}$-modules, that is, coherent sheaves $\sh{F}$ on $X$ equipped with a left $\sh{A}$-actions $\sh{A}\otimes\sh{F}\to\sh{F}$.

We begin with a few facts about ringed spaces. There are adjoint functors $\pi^*\dashv \pi_*\dashv \pi^!$, where $\pi^*=\sh{A}\otimes -:\Coh(X)\to\Coh(\sh{A})$, $\pi_*:\Coh(\sh{A})\to\Coh(X)$ is the forgetful functor, and $\pi^!=\SHom_{\sh{O}_X}(\sh{A},-):\Coh(X)\to\Coh(\sh{A})$. More generally, we have natural isomorphisms
\[
\Hom_{\sh{A}}(\sh{F}\otimes\sh{G},\sh{H})\cong\Hom_{\sh{O}_X}(\sh{G},\SHom_{\sh{A}}(\sh{F},\sh{H})),
\]
for coherent $\sh{A}$-modules $\sh{F},\sh{H}$ and $\sh{O}_X$-module $\sh{G}$, and
\[
\Hom_{\sh{O}_X}(\sh{G}\otimes_{\sh{A}}\sh{F},\sh{H})\cong\Hom_{\sh{A}}(\sh{F},\SHom_{\sh{O}_X}(\sh{G},\sh{H}))
\]
for coherent $\sh{A}$-module $\sh{F}$, $\sh{A}^{\op}$-module $\sh{G}$, and $\sh{O}_X$-module $\sh{H}$.

\subsection{Global dimensions} One defines the global dimension of $\sh{A}$ similarly to the case of algebras. 

\begin{definition}
An $\sh{A}$-module $\sh{F}$ is \emph{locally projective} if for any $x\in X$, there exists an affine open set $U\subset X$ containing $x$ such that $\sh{F}|_U$ is a projective $\sh{A}|_U$-module.
\end{definition}

For any locally free sheaf $\sh{P}$ on $X$, $\sh{A}\otimes\sh{P}$ is naturally a locally projective $\sh{A}$-module. Thus any coherent $\sh{A}$-module $\sh{F}$ admits a resolution by locally projective $\sh{A}$-modules. We define the projective dimension $\pd(\sh{F})$ of $\sh{F}$ to be the shortest length of a projective resolution. Then it is a standard fact in homological algebra to show that

\begin{proposition}
The following two numbers (possibly $\infty$) are the same:
\begin{enumerate}
    \item the supremum of $\pd(\sh{F})$ for all coherent $\sh{A}$-modules $\sh{F}$;
    \item the (co)homological dimension of the category $\Coh(\sh{A})$, that is, the supremum of $n\in\mathbb{N}$ such that $\Ext_{\sh{A}}^n(\sh{F},\sh{G})\neq 0$ for some coherent $\sh{A}$-modules $\sh{F}$ and $\sh{G}$.
\end{enumerate}
We call this number the global dimension of $\sh{A}$, and denote it by $\dim(\sh{A})$.
\end{proposition}

We say $\sh{A}$ is smooth if $\dim(\sh{A})<\infty$. Note that it does not automatically imply $\dim(\sh{A})=\dim(X)$. 

For simplicity, from now on we will assume the sheaf $\sh{A}$ is locally free on $X$. One immediate consequence is that any locally projective (injective) $\sh{A}$-module is locally free (injective) over $\sh{O}_X$. In particular, we have $\dim(\sh{A})\geq\dim(X)$.

Using the local-to-global spectral sequence with some basic properties of non-commutative rings (see for example, \cite[Theorem 4.4]{MR01}), we see that the dimension of $\sh{A}$ can be computed locally using the following lemma.

\begin{lemma}
The dimension of $\sh{A}$ is equal to the supremum of the global dimensions of algebras $\sh{A}_{x}$ for all (closed) points $x\in X$.
\end{lemma}

\subsection{Serre duality} Since $X$ is smooth, the derived category $\der(\Coh(X))$ admits a Serre functor $(-)\otimes\omega_X[n]$, where $\omega_X$ is the dualizing sheaf and $n=\dim(X)$.

\begin{proposition}
If $\sh{A}$ is smooth, then the derived category $\der(\Coh(\sh{A}))$ admits a Serre functor $\omega_{\sh{A}}\ltensor_{\sh{A}}(-)[n]$, where $\omega_{\sh{A}}=\pi^!\omega_X=\SHom_{\sh{O}_X}(\sh{A},\omega_X)$ is the dualizing $\sh{A}$-bimodule. 
\end{proposition}

\begin{proof}
For any perfect complexes $\sh{F}$ and $\sh{G}$ of $\sh{A}$-modules, we have natural isomorphisms
\begin{align*}
    \Hom_{\sh{A}}(\sh{F},\sh{G}) & = \Hom_{\sh{O}_X}(\sh{O}_X,R\SHom_{\sh{A}}(\sh{F},\sh{G}))\\
     & \cong\Hom_{\sh{O}_X}(R\SHom_{\sh{A}}(\sh{F},\sh{G}),\omega_X[n])^*\\
     & \cong\Hom_{\sh{O}_X}(R\SHom_{\sh{A}}(\sh{F},\sh{A})\otimes_{\sh{A}}\sh{G},\omega_X[n])^*\\
     & \cong\Hom_{\sh{A}}(\sh{G},R\SHom_{\sh{A}}(\sh{F},\sh{A})^{\vee}\otimes\omega_X[n])^* \\
     & \cong \Hom_{\sh{A}}(\sh{G},\sh{A}^{\vee}\otimes_{\sh{A}}\sh{F}\otimes\omega_X[n])^* \\
     & = \Hom_{\sh{A}}(\sh{G},\omega_{\sh{A}}\otimes_{\sh{A}}\sh{F}[n])^*.
\end{align*}
\end{proof}

\begin{definition}
We say $\sh{A}$ is Calabi--Yau of dimension $n$ if the derived category $\der(\Coh(\sh{A}))$ is a Calabi--Yau-$n$ category, i.e., the Serre functor is equivalent to  $(-)[n]$.
\end{definition}

In particular if $\sh{A}$ is Calabi--Yau, then $\sh{A}$ is smooth and $\dim(\sh{A})=\dim(X)$.

\begin{proposition}
Suppose $\sh{A}$ is smooth, then the followings are equivalent:
\begin{enumerate}
    \item $\sh{A}$ is Calabi--Yau;
    \item There is an isomorphism $\sh{A}\to\omega_{\sh{A}}$ of $\sh{A}$-bimodules;
    \item There is a non-degenerate bilinear form
    \[
    \sigma:\sh{A}\otimes\sh{A}\to\omega_X
    \]
    of $\sh{O}_X$-modules such that $\sigma$ is symmetric and the diagram
    \[
    \xymatrix{
    \sh{A}\otimes\sh{A}\otimes\sh{A}\ar[r]^{m\otimes\id}\ar[d]_{\id\otimes m} & \sh{A}\otimes\sh{A}\ar[d]^{\sigma}\\
    \sh{A}\otimes\sh{A}\ar[r]_{\sigma} & \omega_X
    }
    \]
    commutes. In other words, $(\sh{A},\sigma)$ is a family of symmetric Frobenius algebras over $X$.
\end{enumerate}
\end{proposition}

\begin{proof}
It is clear that (1) and (2) are equivalent. For (2) and (3), a bilinear form $\sigma:\sh{A}\otimes\sh{A}\to\omega_X$ is non-degenerate if and only if it induces an isomorphism
\[
\sh{A}\to\sh{A}^{\vee}\otimes\omega_X\cong\SHom_{\sh{O}_X}(\sh{A},\omega_X)=\omega_{\sh{A}}
\]
of $\sh{O}_X$-modules. To check it is a morphism of $\sh{A}$-bimodules, we may reduce to affine open sets $\Spec(R)$, where $R$ is regular. Then it is straightforward to verify that two statements are equivalent.
\end{proof}

\begin{example}
Let $X$ be a Calabi--Yau variety. Then any Azumaya algebra $\sh{A}$ on $X$ is Calabi--Yau.
\end{example}

\subsection{Quantum Fermat quintic threefolds}

Let $A$ be a quantum Fermat quintic threefold. To associate a pair $(X,\sh{A})$, we take
\[
B=\mathbb{C}[x_0,\ldots, x_4]\big/\left(\sum_{k=0}^4 x_k\right)
\]
and $X=\Proj(B)\cong\mathbb{P}^3$. Since $A$ is a graded-free $B$-module, the sheaf $\sh{A}$ of non-commutative $\sh{O}_X$-algebras induced by $A$ is locally free. In fact,
\[
\sh{A}=\sh{O}_X\oplus\sh{O}_X(-1)^{\oplus 121}\oplus\sh{O}_X(-2)^{\oplus 381}\oplus\sh{O}_X(-3)^{\oplus 121}\oplus\sh{O}_X(-4)
\]
as a $\sh{O}_X$-module.

It is shown in \cite{Kan15} that the graded algebra $A$ is of finite global dimension, so $\sh{A}$ also has finite global dimension. This point of view gives an alternative proof that $(X,\sh{A})$ is Calabi--Yau.

\begin{lemma}\label{lem:A-is-CY}
The sheaf $\sh{A}$ of $\sh{O}_X$-algebras is Fobenius via
\[
(-,-):\sh{A}\otimes_{\sh{O}_X}\sh{A}\to\sh{A}\to\omega_X,
\]
where the first arrow is the multiplication map, and the second arrow is the projection to the component $\sh{O}_X(-4)\cong\omega_X$. If $\prod_{j} q_{ij}=1$ for all $i$, then the pairing is symmetric.
\end{lemma}

\begin{proof}
We write down the multiplication maps of $\sh{A}$ explicitly. Consider
\[
I=\Big\{\chch{a}=(a_0,a_1,\ldots,a_4)\in\{0,1,\ldots,4\}^5, a_0+a_1+\ldots+a_4\text{ is a multiple of $5$}\Big\},
\]
a basis of $A^{(5)}$ over $B=B^{(5)}$. For simplicity, we will write $a_0+a_1+\ldots+a_4=5\,|\chch{a}|$. Note that $I$ is naturally an abelian group as a subgroup of $(\mathbb{Z}/5\mathbb{Z})^5$ (but the function $|-|$ is not linear). Then as a $\sh{O}_X$-module, we may write
\[
\sh{A}=\bigoplus_{\chch{a}\in I}\sh{O}_X(-|\chch{a}|).
\]
We denote the multiplication map $\sh{A}\otimes\sh{A}\to\sh{A}$ on each component by
\[
\sh{O}_X(-|\chch{a}|)\otimes\sh{O}_X(-|\chch{b}|)\xrightarrow{\phi_{\chch{a},\chch{b}}}\sh{O}_X(-|\chch{a}+\chch{b}|),
\]
where
\[
\phi_{\chch{a},\chch{b}}=q_{\chch{a},\chch{b}}x_0^{c_0}x_1^{c_1}x_2^{c_2}x_3^{c_3}x_4^{c_4}
\]
is the section in $H^0\big(X,\sh{O}_X(|\chch{a}|+|\chch{b}|-|\chch{a}+\chch{b}|)\big)$ given by
\[
q_{\chch{a},\chch{b}}=\prod_{i>j}q_{ij}^{a_i b_j},\quad
c_i=\begin{cases}
5, & a_i+b_i\geq 5;\\
0, & a_i+b_i<5.
\end{cases}
\]

Write $\chch{4}=(4,4,4,4,4)\in I$, $\sh{O}_X(-|\chch{4}|)=\sh{O}_X(-4)$ is the component corresponding to $\omega_X$. Since for each component $\sh{O}_X(-|\chch{a}|)$, there is a unique component $\sh{O}_X(-|\chch{4}-\chch{a}|)$ such that the multiplication map
\[
\sh{O}_X(-|\chch{a}|)\otimes\sh{O}_X(-|\chch{4}-\chch{a}|)\to\sh{O}_X(-|\chch{4}|)\cong\omega_X
\]
is an isomorphism, the induced map $\sh{A}\to\SHom_{\sh{O}_X}(\sh{A},\omega_X)$ is an isomorphism of $\sh{O}_X$-modules.

The pairing $(-,-)$ is symmetric if and only if $q_{\chch{a},\chch{4}-\chch{a}}=q_{\chch{4}-\chch{a},\chch{a}}$ for all $\chch{a}\in I$. That is,
\[
\prod_{i>j} q_{ij}^{a_i(4-a_j)}=\prod_{i>j} q_{ij}^{(4-a_i)a_j}\iff \prod_{i>j} q_{ij}^{a_i-a_j}=1,
\]
which is equivalent to that $\prod_{j}q_{ij}=1$ for all $i$.
\end{proof}

\subsection{Simpson moduli spaces}

We fix a polarization $\sh{O}_X(1)$ on $X$. Stability conditions for coherent $\sh{A}$-modules and their moduli spaces have been studied by Simpson (\cite{Sim94}) for general sheaf $\sh{A}$ of non-commutative algebras on $X$. We recall their definitions and main results.

For a coherent $\sh{A}$-module, we define its Hilbert polynomial, rank, slope, and support to be the same as its underlying coherent sheaf on $X$ (with respect to $\sh{O}_X(1)$). In particular, we say a coherent $\sh{A}$-module is pure (of dimension $d$) if its underlying coherent sheaf is so.

\begin{definition}[\cite{Sim94}]
A coherent $\sh{A}$-module $\sh{F}$ is (semi)stable if it is pure, and for any non-trivial $\sh{A}$-submodule $\sh{G}\subset\sh{F}$,
\[
\frac{p_X(\sh{G})(m)}{r(\sh{G})}\leqp\frac{p_X(\sh{F})(m)}{r(\sh{F})},
\]
for sufficiently large $m$, where $p_X$ is the Hilbert polynomial, and $r$ is the rank.
\end{definition}

It was shown in \cite{Sim94} that all standard facts for semistable sheaves (cf. \cite{HL10}) are also true for semistable $\sh{A}$-modules, such as
\begin{itemize}
    \item Any pure coherent $\sh{A}$-module $\sh{F}$ has a unique filtration, called the \emph{Harder--Narasimhan filtration},
    \[
    0=\sh{F}_0\subset\sh{F}_1\subset\ldots\subset\sh{F}_k=\sh{F}
    \]
    of coherent $\sh{A}$-modules such that the quotients $\sh{F}_i/\sh{F}_{i-1}$'s are semistable $\sh{A}$-modules with strictly decreasing reduced Hilbert polynomials.
    \item Any semistable $\sh{A}$-module $\sh{F}$ has a filtration, called a \emph{Jordan--H\"{o}lder filtration},
    \[
    0=\sh{F}_0\subset\sh{F}_1\subset\ldots\subset\sh{F}_k=\sh{F}
    \]
    of coherent $\sh{A}$-modules such that the quotients $\JH_i:=\sh{F}_i/\sh{F}_{i-1}$'s are stable $\sh{A}$-modules with the same reduced Hilbert polynomial (as $\sh{F}$). Furthermore, the polystable $\sh{A}$-module $\JH(\sh{F}):=\oplus_i\JH_i$ does not depend on the filtration (up to isomorphic). We say two semistable $\sh{A}$-modules $\sh{F}$ and $\sh{G}$ are \emph{S-equivalent} if $\JH(\sh{F})\cong\JH(\sh{G})$ as $\sh{A}$-modules.
    \item If $\sh{A}$ is a stable $\sh{A}$-module, then $\Hom_{\sh{A}}(\sh{F},\sh{F})=\mathbb{C}$.
\end{itemize}

Let $h$ be a polynomial.

\begin{proposition}[\cite{Sim94}]\label{prop:hilb-A}
The Hilbert scheme $\Hilb^h(X,\sh{A})$ parameterizing quotients $\sh{A}\to\sh{F}$ as coherent $\sh{A}$-modules with $p_X(\sh{F})=h$ is representable by a projective scheme. In fact, it is the closed subscheme of the Quot scheme $\Quot_X^p(\sh{A})$ (who parameterizes $\sh{O}_X$-module quotients of $\sh{A}$) given by the locus that the universal quotient is a morphism of $\sh{A}$-modules.
\end{proposition}

The moduli spaces of (semi)stable $\sh{A}$-modules were constructed in the same way as the ones for (semi)stable sheaves, which is via GIT quotient on certain Hilbert (Quot) schemes. We omit the details, and the main results can be stated as follows.

\begin{theorem}[\cite{Sim94}]
Let $\st{M}^{(s)s,h}(X,\sh{A})$ be the moduli stack of (semi)stable $\sh{A}$-modules with Hilbert polynomial $h$. Then
\begin{enumerate}[(a)]
    \item The moduli stack $\st{M}^{(s)s,h}(X,\sh{A})$ is an Artin stack of finite type, and admits a good moduli space $M^{(s)s,h}(X,\sh{A})$.
    \item The coarse moduli scheme $M^{ss,h}(X,\sh{A})$ is projective, whose points are in one-to-one correspondence with S-equivalent classes of semistable $\sh{A}$-modules.
    \item The morphism $\st{M}^{s,h}(X,\sh{A})\to M^{s,h}(X,\sh{A})$ is a $\mathbb{C}^*$-gerbe, and $M^{s,h}(X,\sh{A})$ is the open subscheme of $M^{ss,h}(X,\sh{A})$ whose points corresponds to isomorphism classes of stable $\sh{A}$-modules.
\end{enumerate}
\end{theorem}

We will often simply write $\Hilb^h(\sh{A})$ and $\st{M}^{ss,h}(\sh{A})$ if the base space $X$ is clear.

\section{An obstruction theory for Simpson moduli space}\label{sec:proof_of_atiyah}

For the rest of this paper, our aim is to define DT invariants on a CY3 pair $(X,\sh{A})$. We begin with a study of deformation-obstruction theory of $\sh{A}$-modules and construct an obstruction theory for the moduli space of stable $\sh{A}$-modules. In Section~\ref{sec:define_DT}, we will use it to construct a symmetric obstruction theory and define DT invariants using the Hilbert schemes of $(X,\sh{A})$.

\vspace{1mm}

Let $X$ be a smooth projective variety and $\sh{A}$ a locally free sheaf of non-commutative $\sh{O}_X$-algebras. For a scheme $S$, and a coherent $\sh{A}_{S}$-module $\sh{F}$ on $X\times S$, flat over $S$. Suppose $S\subset\overline{S}$ is a square-zero extension with ideal sheaf $I$.

\begin{definition}
An $\sh{A}$-module extension of $\sh{F}$ over $\overline{S}$ is a coherent $\sh{A}_{\overline{S}}$-module $\overline{\sh{F}}$ on $X\times\overline{S}$, flat over $\overline{S}$, such that $\overline{\sh{F}}|_{X\times S}\cong\sh{F}$ as $\sh{A}_S$-modules.
\end{definition}

It is a general fact (\cite{Low05}) that existence of such extensions must be governed by an obstruction class in $\Ext^2_{\sh{A}_S}(\sh{F},\sh{F}\otimes\pi_S^*I)$. However, to obtain an obstruction theory on the moduli spaces, it requires a more explicit description of the obstruction class. We generalize the result in \cite{HT10}, showing that the obstruction class is the product of Atiyah and Kodiara--Spencer classes.

\begin{theorem}\label{thm:at-cl}
There exists a natural class
\[
\at_{\sh{A}}(\sh{F})\in\Ext^1_{\sh{A}_S}(\sh{F},\sh{F}\otimes\pi^*_S\cpx{L}_S),
\]
called the Atiyah class, such that for any square-zero extension $S\subset\overline{S}$ with ideal sheaf $I$, an $\sh{A}$-module extension of $\sh{F}$ over $\overline{S}$ exists if and only if the obstruction class
\[
\ob=\Big(\sh{F}\xrightarrow{\at_{\sh{A}}(\sh{F})}\sh{F}\otimes\pi_S^*\cpx{L}_S[1]\xrightarrow{\id_{\sh{F}}\otimes\pi_S^*\kappa(S/\overline{S})[1]}\sh{F}\otimes \pi_S^* I[2]\Big)\in\Ext^2_{\sh{A}_S}(\sh{F},\sh{F}\otimes\pi_S^* I)
\]
vanishes, where $\kappa(S/\overline{S})\in\Ext^1_S(S,I)$ is the Kodiara--Spencer class for the extension $S\subset\overline{S}$. Moreover, if an extension of $\sh{F}$ over $\overline{S}$ exists, then all (equivalence classes of) extensions form an affine space over $\Ext^1_{\sh{A}}(\sh{F},\sh{F}\otimes\pi_S^*I)$.
\end{theorem}

We follow closely the method in \cite{HT10}. The key idea is that the obstruction classes are given universally by a morphism of Fourier--Mukai transforms. The proof will be given in the next section. We set up some notations for this chapter.

For any morphism $f:S\to T$ of schemes, we will abuse the notation and write $f$ also for the induced morphism $\id_X\times f:X\times S\to X\times T$. So there are natural functors
\[
f^*:\Coh(\sh{A}_T)\to\Coh(\sh{A}_S),\quad f_*:\Coh(\sh{A}_S)\to\Coh(\sh{A}_T),
\]
where the latter one is induced by the natural morphism $\sh{A}_T=\sh{A}\boxtimes\sh{O}_T\to f_*\sh{A}_S=\sh{A}\boxtimes f_*\sh{O}_S$. We denote by $\der^{(b)}(\sh{A}_S)$ the (bounded) derived category of $\Coh(\sh{A}_S)$. In this section, tensor products $\otimes$ will always be derived tensor products over $\sh{O}$ unless stated otherwise.

We briefly recall the definition of Fourier--Mukai transforms. For any schemes $S$ and $T$, and a complex $\grsh{P}\in\der(\sh{O}_{S\times T})$ of coherent $\sh{O}_{S\times T}$-modules, we may define the Fourier--Mukai functor
\[
\Phi_{\grsh{P}}:\der(\sh{A}_S)\xrightarrow{\pi_S^*}\der(\sh{A}_{S\times T})\xrightarrow{-\otimes\grsh{P}}\der(\sh{A}_{S\times T})\xrightarrow{R\pi_{T*}}\der(\sh{A}_T),
\]
where $\grsh{P}$ is called the Fourier--Mukai kernel. Given any morphism $\phi:\grsh{P}_1\to\grsh{P}_2$ in $\der(\sh{O}_{S\times T})$, it induces a natural transformation $\Phi_{\grsh{P}_1}\to\Phi_{\grsh{P}_2}$ between functors. For any object $\sh{F}\in\der(\sh{A}_S)$, we write
\[
\phi(\sh{F}):\Phi_{\grsh{P}_1}(\sh{F})\to\Phi_{\grsh{P}_2}(\sh{F})
\]
for the induced morphism in $\der(\sh{A}_T)$.

\vspace{2mm}

Now, we use the Atiyah class to construct an obstruction theory on the moduli space. Let $S$ be any scheme, and $\sh{F}$ be a coherent $\sh{A}_S$-module, flat over $S$. Let $\alpha$ be a class in $\Ext^1_{\sh{A}_S}(\sh{F},\sh{F}\otimes\pi_S^*\cpx{L}_S)$. Then $\alpha$ defines a morphism
\[
\alpha:\sh{F}\to\sh{F}\otimes\pi_S^*\cpx{L}_S[1]\quad\text{ in }\der(\sh{A}_S).
\]
Since $\sh{F}$ is flat over $S$, $\sh{F}$ is perfect. The natural map $\sh{F}\to\sh{F}^{\vee\vee}$ is an isomorphism in $\der(\sh{A}_S)$. Thus there is an isomorphism
\[
\sh{F}\otimes\pi_S^*\cpx{L}_S\cong R\SHom_{\sh{O}_X}(\sh{F}^{\vee},\pi_S^*\cpx{L}_S)\quad\text{ in }\der(\sh{A}_S)
\]
By adjunction, $\alpha$ defines a morphism
\[
\sh{F}^{\vee}\otimes_{\sh{A}_S}\sh{F}\to\pi_S^*\cpx{L}_S[1]\quad\text{ in }\der(\sh{O}_{X\times S}).
\]
We then apply Verdier duality, it yields a morphism
\[
R\pi_{S*}\Big((\sh{F}^{\vee}\otimes_{\sh{A}_S}\sh{F})\otimes\pi_X^*\omega_X \Big)[n-1]\to\cpx{L}_S\quad\text{ in }\der(\sh{O}_S).
\]

\begin{lemma}
For any perfect complexes $\sh{F}$ and $\sh{G}$ of coherent $\sh{A}_M$-modules, there is a canonical isomorphism
\[
\Big(R\pi_{M*}R\SHom_{\sh{A}_M}(\sh{F},\sh{G})\Big)^{\vee}\cong R\pi_{M*}\left(\sh{G}^{\vee}\otimes_{\sh{A}_M}\sh{F}\otimes\pi_X^*\omega_X\right)[n].
\]
\end{lemma}

\begin{proof}
Use the same argument as above, with $\pi_S^*\cpx{L}_S$ replaced by $\pi_S^*\sh{O}_S=\sh{O}_{X\times S}$.
\end{proof}

We conclude that any class $\alpha\in\Ext^1_{\sh{A}_S}(\sh{F},\sh{F}\otimes\pi_S^*\cpx{L}_S)$ defines a morphism
\begin{equation}\label{eqn:pre-ob-theory}
    \Big(R\pi_{S*}R\SHom_{\sh{A}_S}(\sh{F},\sh{F})\Big)^{\vee}[-1]\to\cpx{L}_S.
\end{equation}
It is almost by definition to see that if $S=M$ is a \emph{fine} moduli space and $\sh{F}$ is the universal family of coherent $\sh{A}$-modules, then the morphism (\ref{eqn:pre-ob-theory}) induced by a class $\alpha$ is an obstruction theory for $M$ if and only if $\alpha$ is the Atiyah class. 

In general, the moduli space $M:=M^{s,p}_X(\sh{A})$ is not a fine moduli space because the $\mathbb{C}^*$-gerbe $\st{M}\to M$ is not trivial so the universal family on $\st{M}$ does not descend to $M$. We use the fact that any $\mathbb{C}^*$-gerbe is \'{e}tale locally trivial, that is, there is an \'{e}tale cover $U\to X$ with a section
\[
\xymatrix{
 & \st{M}\ar[d] \\
U\ar[r]\ar[ru] & M.
}
\]
We denote by $\mathcal{F}$ the pullback of the universal family of stable $\sh{A}$-modules to $X\times U$.

Consider the natural transformation
\[
\Phi:\Hom_{\Sch}(-,U)\to\sh{M},
\]
where $\sh{M}$ is the moduli functor for stable $\sh{A}$-modules on $X$, and $\Phi$ sends any morphism $f:S\to U$ to the $\sh{A}_T$-module $(\id_X\times f)^*\sh{F}$ on $X\times S$.

\begin{lemma}\label{lem:locally-fine}
The natural transformation $\Phi$ satisfies
\begin{enumerate}
    \item $\Phi(\Spec\mathbb{C})$ is surjective;
    \item For any $f:S\to U$ and any square-zero extension $S\subset\overline{S}$, the map $\Phi(\overline{S})$ induces a bijection between subsets
    \[
    \Big\{\overline{f}:\overline{S}\to U: \overline{f}|_S=f \Big\}\subset \Hom_{\Sch}(\overline{S},U)
    \]
    and
    \[
    \Big\{\text{$\sh{A}$-module extensions of $\sh{F}$ over $\overline{S}$}\Big\}\subset\sh{M}(\overline{S}).
    \]
\end{enumerate}
In other words, $U$ has the same deformation-obstruction theory as a fine moduli space.
\end{lemma}

\begin{proof}
This is essentially the definition of the morphism $U\to\st{M}$ being \'{e}tale.
\end{proof}

\begin{theorem}
Let $U$ and $\sh{F}$ be described as above. Then the Atiyah class $\at_{\sh{A}}(\sh{F})$ defines an obstruction theory
\begin{equation}\label{eqn:ob_theory_U}
    \cpx{E}:=\Big(R\pi_{U*}R\SHom_{\sh{A}_U}(\sh{F},\sh{F})\Big)^{\vee}[-1]\to\cpx{L}_U
\end{equation}
for $U$.
\end{theorem}

\begin{proof}
Let $f:S\to U$ be a morphism, and $S\subset\overline{S}$ be a square-zero extension with ideal sheaf $I$. We consider the class
\[
o=\Big(Lf^*\cpx{E}\to Lf^*\cpx{L}_U\to\cpx{L}_S\xrightarrow{\kappa(S/\overline{S})}I[1]\Big)\in\Ext^1_{\sh{O}_S}(Lf^*\cpx{E},I).
\]
Observe that there are natural isomorphisms
\begin{align*}
    \Hom_{\sh{O}_S}(Lf^*\cpx{E},\cpx{L}_S) & \cong \Hom_{\sh{O}_U}(\cpx{E},Rf_*\cpx{L}_S) \\
    & \cong \Hom_{\sh{O}_{X\times U}}(\sh{F}^{\vee}\otimes_{\sh{A}_U}\sh{F},\pi_U^*Rf_*\cpx{L}_S[1])\\
    & \cong \Hom_{\sh{O}_{X\times U}}(\sh{F}^{\vee}\otimes_{\sh{A}_U}\sh{F},Rg_*\pi_S^*\cpx{L}_S[1])\\
    & \cong \Hom_{\sh{O}_{X\times S}}(g^*(\sh{F}^{\vee}\otimes_{\sh{A}_U}\sh{F}),\pi_S^*\cpx{L}_S[1]) \\
    & \cong \Hom_{\sh{O}_{X\times S}}(g^*\sh{F}^{\vee}\otimes_{g^*\sh{A}_U}g^*\sh{F},\pi_S^*\cpx{L}_S)\\
    & \cong \Hom_{\sh{O}_{X\times S}}((g^*\sh{F})^{\vee}\otimes_{\sh{A}_S}(g^*\sh{F}),\pi_S^*\cpx{L}_S) \\
    & \cong \Hom_{\sh{A}_S}(g^*\sh{F},g^*\sh{F}\otimes\pi_S^*\cpx{L}_S)
\end{align*}
where $g=\id_X\times f:X\times S\to X\times U$. By the functoriality of Atiyah classes, it sends the composition $(Lf^*\cpx{E}\to Lf^*\cpx{L}_U\to\cpx{L}_S)$ to the Atiyah class $\at_{\sh{A}}(g^*\sh{F})$ on $X\times S$. Therefore the class $o\in\Ext^1_{\sh{O}_S}(Lf^*\cpx{E},I)$ corresponds to the obstruction class $\ob\in\Ext^2_{\sh{A}_S}(g^*\sh{F},g^*\sh{F}\otimes\pi_S^*\cpx{L}_S)$ for the coherent $\sh{A}_S$-module $g^*\sh{F}$ on $X\times S$.

The rest of the proof follows from Lemma~\ref{lem:locally-fine}: there exists an extension of $f:S\to U$ to $\overline{S}$ if and only if there exists an $\sh{A}$-module deformation of $g^*\sh{F}$ over $\overline{S}$.
\end{proof}

The coherent $\sh{A}_U$-module $\sh{F}$ on $X\times U$ is pulled back from the $\mathbb{C}^*$-gerbe $\st{M}\to M$, so it can be regarded as a \emph{twisted} $\sh{A}_M$-module on $X\times M$.

\begin{lemma}
The complex $R\SHom_{\sh{A}_{U}}(\sh{F},\sh{F})$ on $X\times U$ descends to a complex in $\der_{\text{\'{e}t}}(\sh{O}_{X\times M})$, which we denote by $R\SHom_{\sh{A}_M}(\sh{F},\sh{F})$.
\end{lemma}

\begin{proof}
Note that for any coherent $\sh{A}$-module $\sh{G}$, $\SHom_{\sh{A}}(\sh{G},\sh{G})$ is the equalizer of
\[
\xymatrix{
\SHom_{\sh{O}_X}(\sh{G},\sh{G})\ar@<-.5ex>[r] \ar@<.5ex>[r] & \SHom_{\sh{O}_X}(\sh{A}\otimes\sh{G},\sh{G}).
}
\]
Then the proof follows from the fact that for any twisted sheaf $\sh{F}$, $\SHom_{\sh{O}_X}(\sh{F},\sh{F})$ is a (untwisted) coherent sheaf (see for instance \cite{Cal00}).
\end{proof}

By functoriality of the Atiyah class, the obstruction theory $\cpx{E}\to\cpx{L}_U$ also descends to a morphism in $\der_{\text{\'{e}t}}(\sh{O}_M)$
\begin{equation}\label{eqn:ob_theory}
    \Big(R\pi_{M*}R\SHom_{\sh{A}_M}(\sh{F},\sh{F})\Big)^{\vee}[-1]\to\cpx{L}_M\quad\text{ in }\der_{\text{\'{e}t}}(\sh{O}_M).
\end{equation}
Since being an obstruction theory is an \'{e}tale local property, we conclude that

\begin{corollary}\label{cor:ob_for_M}
The morphism (\ref{eqn:ob_theory}) is an obstruction theory for the moduli space $M$.
\end{corollary}

\section{Proof of Theorem~\ref{thm:at-cl}}

We fix a scheme $S$ and a square-zero extension $i:S\subset\overline{S}$ with ideal sheaf $I$. Let $\sh{F}$ be a coherent $\sh{A}_S$-module on $X\times S$, flat over $S$. Suppose $\overline{\sh{F}}$ is an $\sh{A}$-module extension of $\sh{F}$ over $\overline{S}$. The isomorphism $Li^*\overline{\sh{F}}=i^*\overline{\sh{F}}\to\sh{F}$ induces an exact triangle
\[
Ri_*(\sh{F}\otimes\pi_S^*I)\to\overline{\sh{F}}\to Ri_*\sh{F}\xrightarrow{e} Ri_*(\sh{F}\otimes\pi_S^*I)[1],
\]
which gives a class $e\in\Ext^1_{\sh{A}_{\overline{S}}}(Ri_*\sh{F},Ri_*(\sh{F}\otimes\pi_S^*I))$. For any $\sh{F}$, we have an exact triangle
\begin{equation}\label{eqn:exact_Q_adj}
    Q_{\sh{F}}\to Li^* Ri_*\sh{F}\to \sh{F}
\end{equation}
in $\der(\sh{A}_S)$ given by adjunction.

\begin{lemma}
A class $e\in\Ext^1_{\sh{A}_{\overline{S}}}(Ri_*\sh{F},Ri_*(\sh{F}\otimes\pi_S^*I))$ is given by an $\sh{A}$-module deformations of $\sh{F}$ over $\overline{S}$ if and only if the composition
\[
\Phi_e: Q_{\sh{F}}\to Li^* Ri_*\sh{F}\xrightarrow{Li^*e}Li^* Ri_*(\sh{F}\otimes\pi_S^*I)
\]
is an isomorphism in $\der(\sh{A}_S)$.
\end{lemma}

\begin{proof}
Observe the diagram
\begin{equation}
    \xymatrix{
     & & Q_{\sh{F}}\ar[d]\ar[rd]^{\Phi_e} & \\
    Li^*Ri_*(\sh{F}\otimes\pi_S^*I)\ar[r] & Li^*\grsh{F}\ar[r]\ar[rd]_{r} & Li^*Ri_*\sh{F}\ar[r]\ar[d] & Li^*Ri_*(\sh{F}\otimes\pi_S^*I)[1]\\
     & & \sh{F} & 
    }
\end{equation}
in $\der(\sh{A}_S)$. The morphism $r:Li^*\grsh{F}\to\sh{F}$ is an isomorphism if and only if $\Phi_e$ is. \emph{I think there is some standard argument to show that such complex $\grsh{F}$ must be concentrated in degree zero...}
\end{proof}

Now, we apply $\Hom_{\sh{A}_S}(-,\sh{F}\otimes\pi_S^*I)$ to the exact triangle (\ref{eqn:exact_Q_adj}), it yields
\[
\hspace{-13mm}
\xymatrix@=1em{
 & \Ext^1_{\sh{A}_{\overline{S}}}(Ri_*\sh{F},Ri_*(\sh{F}\otimes\pi_S^*I))\ar@{=}[d] & & \\
\Ext^1_{\sh{A}_S}(\sh{F},\sh{F}\otimes\pi_S^*I)\ar[r] & \Ext^1_{\sh{A}_S}(Li^*Ri_*\sh{F},\sh{F}\otimes\pi_S^*I)\ar[r] & \Ext^1_{\sh{A}_S}(Q_{\sh{F}},\sh{F}\otimes\pi_S^*I)\ar[r]^{\delta} & \Ext^2_{\sh{A}_S}(\sh{F},\sh{F}\otimes\pi_S^*I)
}
\]
The second arrow sends a class $e$ to the morphism
\[
\Psi_e: Q_{\sh{F}}\xrightarrow{\Phi_e}Li^*Ri_*(\sh{F}\otimes\pi_S^*I)[1]\to \sh{F}\otimes\pi_S^*I[1],
\]
where the second map is the adjunction map. The proof consists of following steps.
\begin{enumerate}[(a)]
    \item There exists a class $\pi_{\sh{F}}\in\Ext^1_{\sh{A}_S}(Q_{\sh{F}},\sh{F}\otimes\pi_S^*I_S)$ such that $\Psi_e=\pi_{\sh{F}}$ for any class $e$ given by a deformation.
    \item The obstruction class $\delta(\pi_{\sh{F}})\in\Ext^2_{\sh{A}_S}(\sh{F},\sh{F}\otimes\pi_S^*I)$ is the product of Atiyah class and Kodaira--Spencer class.
    \item If the obstruction class $\delta(\pi_{\sh{F}})$ vanishes, then there exists a class $e$ given by a deformation such that $\Psi_e=\pi_{\sh{F}}$. 
    \item Suppose $e$ and $e'$ are two classes such that $\Psi_e=\Psi_{e'}$. If a class $e$ is given by a deformation, then so is $e'$.
\end{enumerate}

\subsection*{Construction of $\pi_{\sh{F}}$}

We first consider the trivial case $\sh{A}=\sh{O}_X$, $\sh{F}=\sh{O}_{X\times S}$, and the class $e$ given by the deformation $\sh{O}_{X\times\overline{S}}$. Then the morphism $\Psi_e$ is determined by the square-zero extension $S\subset\overline{S}$, which we denote by
\begin{equation}\label{eqn:pi-map}
    \pi:Q_{\sh{O}_{X\times S}}\to Li^*Ri_*\pi_S^*I[1]\to\pi_S^*I[1].
\end{equation}
Note that it is the pull-back of the morphism $Q_{\sh{O}_S}\to I[1]$ in $\der(\sh{O}_S)$ defined in the same way.

Assume that an $\sh{A}$-module deformation $\overline{\sh{F}}$ of $\sh{F}$ over $\overline{S}$ exists. Then for any coherent sheaf $\sh{P}$ on $X\times S$, there are canonical isomorphisms
\[
\sh{F}\otimes Li^*Ri_*\sh{P}\cong Li^*(\overline{\sh{F}}\otimes Ri_*\sh{P})\cong Li^*Ri_*(\sh{F}\otimes\sh{P}).
\]
This implies that $\Psi_e$ is equal to $\pi_{\sh{F}}:=\sh{F}\otimes\pi$, up to a canonical isomorphism $Q_{\sh{F}}\cong\sh{F}\otimes Q_{\sh{O}_{X\times S}}$. 

\subsection*{The obstruction class $\delta(\pi_{\sh{F}})$}

Observe that $Q_{\sh{F}}$ and $\sh{F}\otimes Q_{\sh{O}_{X\times S}}$ may not be isomorphic for general coherent $\sh{A}$-module $\sh{F}$, so we need an alternative definition of $\pi_{\sh{F}}$. To define $\pi_{\sh{F}}$ and study the obstruction class $\delta(\pi_{\sh{F}})$, we recall several facts proved in \cite{HT10}.

\begin{enumerate}[(i)]
    \item There exists an exact triangle
    \[
    \grsh{Q}\to \grsh{H}\to \Delta_*\sh{O}_S\xrightarrow{\delta_0}\grsh{Q}[1]\quad\text{ in }\der(\sh{O}_{S\times S}),
    \]
    where $\Delta:S\to S\times S$ is the diagonal map, such that the exact triangle (\ref{eqn:exact_Q_adj}) is given by Fourier--Mukai transforms
    \[
    \Phi_{\pi_{S\times S}^*\grsh{Q}}(\sh{F})\to\Phi_{\pi_{S\times S}^*\grsh{H}}(\sh{F})\to \sh{F}\xrightarrow{(\pi_{S\times S}^*\delta_0)(\sh{F})}\Phi_{\pi_{S\times S}^*\grsh{Q}}(\sh{F})[1].
    \]
    In particular, the map $\delta:\Ext_{\sh{A}}^1(Q_{\sh{F}},\sh{F}\otimes\pi_S^*I)\to\Ext^2_{\sh{A}}(\sh{F},\sh{F}\otimes\pi_S^*I)$ is the composition with $(\pi_{S\times S}^*\delta_0)(\sh{F})$.

    \item There exists a natural morphism
    \[
    \pi_0:\grsh{Q}\to \Delta_*I[1]\quad\text{ in }\der(\sh{O}_{S\times S})
    \]
    such that $(\pi_{S\times S}^*\pi_0)(\sh{F})=\pi_{\sh{F}}$ for any $\sh{F}$ admitting a deformation. Therefore we define $\pi_{\sh{F}}=(\pi_{S\times S}^*\pi_0)(\sh{F})$ for general coherent $\sh{A}_S$-module $\sh{F}$.

    \item The \emph{universal obstruction class}
    \[
    \omega_0:=\pi_0[1]\circ\delta_0:\Delta_*\sh{O}_S\to\Delta_*I[2]\quad\text{ in }\der(\sh{O}_{S\times S})
    \]
    decomposes into
    \[
    \omega_0:\Delta_*\sh{O}_S\xrightarrow{\alpha_S}\Delta_*\cpx{L}_S[1]\xrightarrow{\Delta_*\kappa(S/\overline{S})[1]}\Delta_*I[2],
    \]
    where $\alpha_S\in\Ext^1_{S\times S}(\Delta_*\sh{O}_S,\Delta_*\cpx{L}_S)$ is the \emph{(truncated) universal Atiyah class}, which is intrinsic and functorial to $S$, and $\kappa(S/\overline{S})\in\Ext^1_{S}(\cpx{L}_S,I)$ is the (truncated) Kodaira--Spencer class of the square-zero extension $S\subset\overline{S}$. (See \cite{HT10}, Definition 2.3 and 2.7 for details.)
\end{enumerate}

\begin{definition}
We define the Atiyah class for a coherent $\sh{A}_S$-module $\sh{F}$ to be
\[
(\pi_{S\times S}^*\alpha_S)(\sh{F})\in \Ext^1_{\sh{A}_S}(\sh{F},\sh{F}\otimes\pi_S^*\cpx{L}_S).
\]
\end{definition}

The proofs can be found in \cite[Section 2.5 and Section 3.1]{HT10}. We remark that while \cite{HT10} only consider coherent sheaves, the proofs are all done at the level of Fourier--Mukai kernels. Therefore it also works for any coherent $\sh{A}$-modules.

\subsection*{Existence of deformations}

We will show that if the obstruction class $\delta(\pi_{\sh{F}})=(\pi_{S\times S}^*\omega_0)(\sh{F})$ vanishes, then there exists an $\sh{A}$-module extension of $\sh{F}$ over $\overline{S}$.

First we remark that we may assume both $X$ and $S$ are affine as in \cite{HT10}, Section 3.3. Although the existence of deformations is \emph{not} a local property, for any given class $e\in\Ext^1_{\sh{A}_S}(Li^*Ri_*\sh{F},\sh{F}\otimes\pi_S^*I)$, the conditions that $\Phi_e$ being an isomorphism and $\Psi_e=\pi_{\sh{F}}$ can be checked locally. In other words, if $\delta(\pi_{\sh{F}})=0$ but there is no deformation of $\sh{F}$, then there is a class $e$ such that $\Psi_e=\pi_{\sh{F}}$ but $e$ is not given by a deformation. Then we can find an affine open set $U$ such that $\Phi_e|_U$ is not an isomorphism but $\Psi_e|_U=\pi_{\sh{F}}|_U$, which is a contraction.

Therefore we may assume that $X=\Spec(B)$, $\sh{A}=A$ is an $R$-algebra, $S=\Spec(R)$, and $\overline{S}=\Spec(\overline{R})$. Let $M$ be a (left) $A_R:=A\otimes_{\mathbb{C}} R$-module, flat over $R$. For convenience, a tensor product $\otimes$ without subscript is over $R$. We first recall the standard obstruction theory for modules (cf. \cite{Lau79}).

Choose a (possibly-infinite) free resolution of $M$
\[
\ldots\to A_R^{\oplus n_{-3}}\xrightarrow{d_{-3}} A_R^{\oplus n_{-2}}\xrightarrow{d_{-2}} A_R^{\oplus n_{-1}}\to M\to 0.
\]
Consider the trivial deformation $A_{\overline{R}}^{\oplus n_{\bullet}}$ of $A_R^{\oplus n_{\bullet}}$. Then we choose an arbitrary lifting $d'_{\bullet}:A^{\oplus n_{\bullet}} _{\overline{R}}\to A^{\oplus n_{\bullet+1}}_{\overline{R}}$ of $d_{\bullet}$. Since $(d'_{\bullet+1}\circ d'_{\bullet})|_R=d_{\bullet+1}\circ d_{\bullet}=0$, the map $d'_{\bullet+1}\circ d'_{\bullet}$ factors into
\[
d'_{\bullet+1}\circ d'_{\bullet}:A_{\overline{R}}^{\oplus n_{\bullet}}\to A_{R}^{\oplus n_{\bullet}}\xrightarrow{\ob_{\bullet}} A_R^{\oplus n_{\bullet +2}}\otimes I\to A_{\overline{R}}^{\oplus n_{\bullet +2}},
\]
where the first and third arrows are given by the extension $0\to I\to \overline{R}\to R\to 0$. It is well-known that the class
\begin{equation}\label{eqn:standard-ob}
    \{\ob_{\bullet}\}\in\Hom_{A_R}(A_R^{\oplus n_{\bullet}},A_R^{\oplus n_{\bullet+2}}\otimes I)
\end{equation}
is a $2$-cocycle defining a class in $\Ext^2_{A_R}(M,M)$ which is independent of the choice of the resolution $(A_R^{\oplus n_{\bullet}},d_{\bullet})$ and lifting $d'_{\bullet}$. We will show that this obstruction class is the same as $\delta(\pi_{\sh{F}})$.

Recall that the universal obstruction class $\omega_0\in\Ext^2_{S\times S}(\Delta_*\sh{O}_S,\Delta_*I)$ is represented by the $2$-extension
\[
0 \to \Delta_*I\to \grsh{J}|_{S\times S}\to \sh{O}_{S\times S}\to\Delta_*\sh{O}_S\to 0,
\]
where $\grsh{J}$ is the ideal sheaf defining $\overline{S}\subset \overline{S}\times\overline{S}$. We omit the details but a crucial consequence is that if $M=A_R^{\oplus n}$ is free, then the obstruction class $(\pi_{S\times S}^*\omega_0)(M)\in\Ext^2_{A_R}(M,M\otimes I)$ is represented by the $2$-extension (of $A_R$-modules)
\[
0\to A_R^{\oplus n}\otimes I\to K^{\oplus n}|_R\to \Gamma^{\oplus n}|_R\to A_R^{\oplus n}\to 0,
\]
where the restriction $-|_R$ is the tensor product $-\otimes_{\overline{R}} R$, $\Gamma=A_{\overline{R}}\otimes_{\mathbb{C}} R$ is the free $A_{\overline{R}}$-module, the arrow
\[
\Gamma=A_{\overline{R}}\otimes_{\mathbb{C}} R= A\otimes_{\mathbb{C}}\overline{R}\otimes_{\mathbb{C}}R\to A_R=A\otimes_{\mathbb{C}} R
\]
is induced by the $\overline{R}$-linear evaluation map $\overline{R}\otimes_{\mathbb{C}} R\to R$ via $\overline{R}\to R$, and $K=\ker(\Gamma\to A_R)$ is the kernel. Since $\overline{R}\otimes_{\mathbb{C}}R$ is a free $\overline{R}$-module, we may choose a (non-canonical) splitting
\[
\overline{R}\otimes_{\mathbb{C}} R\cong L\oplus\overline{R}
\]
such that the evaluation map $\overline{R}\otimes_{\mathbb{C}} R$ is given by the short exact sequence
\[
0\to L\oplus I\to L\oplus\overline{R}\to R\to 0. 
\]
Then $\Gamma\cong N\oplus A_{\overline{R}}$, $K\cong N\oplus (A_R\otimes I)$, where $N=A\otimes_{\mathbb{C}} L$ is a free $A_{\overline{R}}$-module.

Let $A_R^{\oplus n_{\bullet}}\to M$ be a free resolution. Then it associates a $2$-extension
\begin{equation}\label{eqn:2-ext-complex}
    0 \to A_R^{\oplus n_{\bullet}}\otimes I\to K^{\oplus n_{\bullet}}|_R\to \Gamma^{\oplus n_{\bullet}}|_R\to A_R^{\oplus n_{\bullet}}\to 0
\end{equation}
of $A_R$-modules, and a short exact sequence 
\[
0\to K^{\oplus n_{\bullet}}\to \Gamma^{\oplus n_{\bullet}}\to A_R^{\oplus n_{\bullet}}\to 0
\]
of complexes of $A_{\overline{R}}$-modules, where the differentials are arbitrarily chosen lifting of the differentials in (\ref{eqn:2-ext-complex}).  We write down the differentials explicitly with respect to the splitting $\Gamma\cong N\oplus A_{\overline{R}}$ and $K\cong N\oplus (A_R\otimes I)$:
\[
\xymatrix{
0\ar[r] & N^{\oplus n_{\bullet}}\oplus (A_R\otimes I)^{\oplus n_{\bullet}}\ar[r]\ar[d]_{\begin{psmallmatrix}\ast & \ast \\ \eta_{\bullet} & \ast \end{psmallmatrix}} & N^{\oplus n_{\bullet}}\oplus A_{\overline{R}}^{\oplus n_{\bullet}}\ar[r]\ar[d]_{\begin{psmallmatrix}\ast & \gamma_{\bullet} \\ \eta_{\bullet} & d'_{\bullet} \end{psmallmatrix}} & A_R^{\oplus n_{\bullet}}\ar[r]\ar[d]^{d_{\bullet}} & 0\\
0\ar[r] & N^{\oplus n_{\bullet+1}}\oplus (A_R\otimes I)^{\oplus n_{\bullet+1}}\ar[r] & N^{\oplus n_{\bullet+1}}\oplus A_{\overline{R}}^{\oplus n_{\bullet+1}}\ar[r] & A_R^{\oplus n_{\bullet+1}}\ar[r] & 0
}
\]
and
\[
\hspace{-8mm}
\xymatrix@C=1em{
0\ar[r] & (A_R\otimes I)^{\oplus n_{\bullet}}\ar[r]\ar[d] & N^{\oplus n_{\bullet}}|_R\oplus (A_R\otimes I)^{\oplus n_{\bullet}}\ar[r]\ar[d]_{\begin{psmallmatrix}\ast & \ast \\ \sigma_{\bullet} & \ast \end{psmallmatrix}} & N^{\oplus n_{\bullet}}|_R\oplus A_R^{\oplus n_{\bullet}}\ar[r]\ar[d]_{\begin{psmallmatrix}\ast & \beta_{\bullet} \\ \ast & d_{\bullet} \end{psmallmatrix}} & A_R^{\oplus n_{\bullet}}\ar[r]\ar[d]^{d_{\bullet}} & 0\\
0\ar[r] & (A_R\otimes I)^{\oplus n_{\bullet+1}}\ar[r] & N^{\oplus n_{\bullet+1}}|_R\oplus (A_R\otimes I)^{\oplus n_{\bullet+1}}\ar[r] & N^{\oplus n_{\bullet+1}}|_R\oplus A_R^{\oplus n_{\bullet+1}}\ar[r] & A_R^{\oplus n_{\bullet+1}}\ar[r] & 0
}
\]
Observe that
\[
\left\{A_R^{\oplus n_{\bullet}}\xrightarrow{\beta_{\bullet}} N^{\oplus n_{\bullet+1}}|_R\xrightarrow{\sigma_{\bullet+1}} (A_R\otimes I)^{\oplus n_{\bullet+2}}\right\}\in\Hom_{A_R}(A_R^{\oplus n_{\bullet}},A_R^{\oplus n_{\bullet+2}}\otimes I)
\]
defines a $2$-cocycle which corresponds to the class of the $2$-extension (\ref{eqn:2-ext-complex}) in $\Ext^2_{A_R}(M,M\otimes I)$, i.e., the obstruction class $\delta(\pi_{\sh{F}})$.

On the other hand, $d'_{\bullet}:A_{\overline{R}}^{\oplus n_{\bullet}}\to A_{\overline{R}}^{\oplus n_{\bullet+1}}$ is a lifting of $d_{\bullet}$. The differentials on $\Gamma^{\oplus n_{\bullet}}$ implies that
\[
-d'_{\bullet+1}\circ d'_{\bullet}=\eta_{\bullet+1}\circ\gamma_{\bullet}: A_{\overline{R}}^{\oplus n_{\bullet}}\to N^{\oplus n_{\bullet+1}}\to A_{\overline{R}}^{\oplus n_{\bullet+2}}.
\]
Since $\eta_{\bullet+1}$ factors through $N^{\oplus n_{\bullet+1}}\to (A_R\otimes I)^{\oplus n_{\bullet+2}}\to A_{\overline{R}}^{\oplus n_{\bullet+2}}$, the composition can be decomposed into
\[
\eta_{\bullet+1}\circ\gamma_{\bullet}:A_{\overline{R}}^{\oplus n_{\bullet}}\to A_R^{\oplus n_{\bullet}}\xrightarrow{\gamma_{\bullet}|_R} N^{\oplus n_{\bullet+1}}|_R\xrightarrow{\eta_{\bullet+1}|_R}(A_R\otimes I)^{\oplus n_{\bullet+2}}\to A_{\overline{R}}^{\oplus n_{\bullet+2}}.
\]
By definition, $\gamma_{\bullet}|_R=\beta_{\bullet}$ and $\eta_{\bullet+1}|_R=\sigma_{\bullet+1}$. This shows that the classical obstruction class (\ref{eqn:standard-ob}) defined by the lifting $d'_{\bullet}$ is the class $-\delta(\pi_{\sh{F}})$.

Finally, suppose $e$ is the class corresponding to a deformation $(A_{\overline{R}}^{\oplus n_{\bullet}},d'_{\bullet})$, and $e'$ is another class such that $\Psi_e=\Psi_{e'}$. Then $e-e'$ is in the image of a $1$-cocycle
\[
\{f_{\bullet}\}\in\Hom_{A_R}(A_R^{\oplus n_{\bullet}},A_R^{\oplus n_{\bullet+1}}\otimes I).
\]
Then it is a standard fact that $(A_{\overline{R}}^{\oplus n_{\bullet}},d'_{\bullet}+\tilde{f}_{\bullet})$ also defines a deformation, where
\[
\tilde{f}_{\bullet}:A_{\overline{R}}^{\oplus n_{\bullet}}\to A_R^{\oplus n_{\bullet}}\xrightarrow{f_{\bullet}}A_R^{\oplus n_{\bullet+1}}\to A_{\overline{R}}^{\oplus n_{\bullet+1}},
\]
which then corresponds to the class $e'$.

\section{Donaldson--Thomas invariants}\label{sec:define_DT}

\subsection{Symmetric obstruction theories}

We now restrict our attention to $(X,\sh{A})$ being CY3. Let $M$ be a quasi-projective coarse moduli scheme of stable $\sh{A}$-modules with a universal twisted $\sh{A}_M$-module $\sh{F}$ on $X\times M$.

First we construct a symmetric bilinear form
\[
\theta:R\pi_{M*}R\SHom_{\sh{A}_M}(\sh{F},\sh{F})\to \Big(R\pi_{M*}R\SHom_{\sh{A}_M}(\sh{F},\sh{F})\Big)^{\vee}[1].
\]
We write $\cpx{F}=R\SHom_{\sh{A}_M}(\sh{F},\sh{F})$. For a given isomorphism $\omega_{\sh{A}}\to\sh{A}$ of $\sh{A}$-bimodules, it induces an isomorphism
\[
R\SHom_{\sh{A}_M}(\sh{F},\omega_{\sh{A}}\otimes_{\sh{A}}\sh{F})\to\cpx{F}\quad\text{ in }\der(\sh{O}_{X\times M}).
\]
Taking $(-)^{\vee}[-1]$ on both sides, it gives an isomorphism
\begin{equation}\label{eqn:sym-bilinear}
    \left(R\pi_{M*}\cpx{F}\right)^{\vee}[-1]\to\Big(R\pi_{M*}R\SHom_{\sh{A}_M}(\sh{F},\omega_{\sh{A}}\otimes_{\sh{A}}\sh{F})\Big)^{\vee}[-1]\quad\text{ in }\der(\sh{O}_M),
\end{equation}
and the right hand side is isomorphic to
\[
R\pi_{M*}\Big(\big(\pi_X^*\omega_{\sh{A}}\otimes_{\sh{A}_M}\sh{F}\big)^{\vee}\otimes_{\sh{A}_M}\sh{F}\otimes\pi_X^*\omega_X\Big)[2].
\]
Note that
\begin{align*}
    \big(\pi_X^*\omega_{\sh{A}}\otimes_{\sh{A}_M}\sh{F}\big)^{\vee} &=R\SHom_{\sh{O}_{X\times M}}(\pi_X^*\omega_{\sh{A}}\otimes_{\sh{A}_M}\sh{F},\sh{O}_{X\times M})\\
     &\cong R\SHom_{\sh{A}_M}(\sh{F},R\SHom_{\sh{O}_{X\times M}}(\pi_X^*\omega_{\sh{A}},\sh{O}_{X\times M}))\\
     &\cong R\SHom_{\sh{A}_M}(\sh{F},\sh{A}_M\otimes\pi_X^*\omega_X^{\vee}) & \text{ in }\der(\sh{A}_M^{\op}),
\end{align*}
where the last isomorphism is given by $\pi_X^*\omega_{\sh{A}}\cong\SHom_{\sh{O}_{X\times M}}(\pi_X^*\sh{A},\pi_X^*\omega_X)\cong\sh{A}_M^{\vee}\otimes\pi_X^*\omega_X$.
Therefore
\[
\big(\pi_X^*\omega_{\sh{A}}\otimes_{\sh{A}_M}\sh{F}\big)^{\vee}\otimes_{\sh{A}_M}\big(\sh{F}\otimes\pi_X^*\omega_X\big)\cong R\SHom_{\sh{A}_M}(\sh{F},(\sh{F}\otimes\pi_X^*\omega_X)\otimes\pi_X^*\omega_X^{\vee})\cong\cpx{F}.
\]
Thus (\ref{eqn:sym-bilinear}) is an isomorphism
\[
\theta:(R\pi_{M*}\cpx{F})^{\vee}[-1]\to R\pi_{M*}\cpx{F}\,[2]=\Big((R\pi_{M*}\cpx{F})^{\vee}[-1]\Big)^{\vee}[1].
\]
In fact, we have $\theta^{\vee}[1]=\theta$, so $\theta$ is symmetric.

\begin{theorem}\label{thm:sym_ob}
The moduli scheme $M$ carries a symmetric obstruction theory
\[
\Big(\tau^{[1,2]}R\pi_{M*}R\SHom_{\sh{A}_M}(\sh{F},\sh{F})\Big)^{\vee}[-1]\to\cpx{L}_M
\]
which in particular is a perfect obstruction theory.
\end{theorem}

\begin{proof}
Consider the map $\sigma:\sh{O}_M\to R\pi_{M*}\cpx{F}$ induced by the scalar map $R\pi_M^*\sh{O}_M=\sh{O}_{X\times M}\to\cpx{F}$. For any point $m\in M$, this map induces the scalar map $\sigma_m:\mathbb{C}\to\Hom_{\sh{A}}(\sh{F}_m,\sh{F}_m)$, which is an isomorphism since $\sh{F}_m$ is stable. Thus the cone of $\sigma$ is the truncated complex $\tau^{\geq 1}R\pi_{M*}\cpx{E}$.

On the other hand, we may consider 
\[
R\pi_{M*}\cpx{F}\xrightarrow{\theta^{-1}[-2]}(R\pi_{M*}\cpx{F})^{\vee}[-3]\xrightarrow{\sigma^{\vee}[-3]}\sh{O}_M[-3].
\]
The induced map $\Ext^3_{\sh{A}}(\sh{F}_m,\sh{F}_m)\to\mathbb{C}$ on each point $m\in M$ is dual to the scalar map, which is an isomorphism. The cone of the composition $\tau^{\geq 1}R\pi_{M*}\cpx{F}\to\sh{O}_M[-3]$ is the truncated complex $\tau^{[1,2]}R\pi_{M*}\cpx{F}$.

Therefore, $\tau^{[1,2]}R\pi_{M*}\cpx{F}$ is perfect of amplitude in degree $1$ and $2$, and the obstruction theory $(R\pi_{M*}\cpx{F})^{\vee}[-1]\to\cpx{L}_M$ induces a morphism
\[
(\tau^{[1,2]}R\pi_{M*}\cpx{F})^{\vee}[-1]\to\cpx{L}_M.
\]
Furthermore, by the construction, $\theta$ induces a symmetric bilinear form 
\[
\theta:(\tau^{[1,2]}R\pi_{M*}\cpx{F})^{\vee}[-1]\to\big((\tau^{[1,2]}R\pi_{M*}\cpx{F})^{\vee}[-1]\big)^{\vee}[1].
\]
\end{proof}

If the moduli space $M=M^{s,h}(X,\sh{A})$ is projective (for example, when the Hilbert polynomial $h$ has coprime coefficients), then we may define the DT invariant via integration
\[
\int_{[M]_{\vir}}1,
\]
which equals to the Behrend function weighted Euler characteristic $\chi(M,\nu_M)$ by \cite{Beh09}. In particular, this invariant depends only on the scheme structure of the moduli space $M$, which depends only on the abelian category $\Coh(\sh{A})$ with a chosen stability condition.

\subsection{Donaldson--Thomas invariants}

In classical Donaldson--Thomas theory, the Hilbert schemes (with a fixed curve class) are isomorphic to the moduli spaces of stable sheaves with fixed determinant. Therefore one may study Donaldson--Thomas invariants on Hilbert schemes. We do not have the notion of determinant for $\sh{A}$-modules, but we can still define Donaldson--Thomas type invariants on Hilbert schemes under some suitable assumptions. 

\begin{lemma}\label{lem:dt-on-hilb}
Let $h$ be a polynomial of degree $\leq 1$. Suppose the stalk $\sh{A}_{\eta}$ at generic point $\eta\in X$ is a division algebra, then there is a natural morphism
\begin{equation}\label{eqn:hilb2M}
    \Hilb^h(\sh{A})\to M^{s,p-h}(\sh{A}),
\end{equation}
sending any quotient $\sh{A}\to\sh{F}$ to its kernel, where $p$ is the Hilbert polynomial of $\sh{A}$.
\end{lemma}

\begin{proof}
This is the analogue of the classical result that all torsion-free rank $1$ sheaves are stable. Since $\sh{A}_{\eta}$ is a division algebra, any $\sh{A}$-submodule of $\sh{A}$ has the same rank as $\sh{A}$.
\end{proof}

\begin{proposition}\label{prop:dt-on-hilb}
Under the assumptions in Lemma~\ref{lem:dt-on-hilb}, if $H^1(X,\sh{A})=0$, then the natural morphism (\ref{eqn:hilb2M}) is an open immersion.
\end{proposition}

\begin{proof}
Let $0\to\sh{I}\to\sh{A}\to\sh{F}\to 0$ be a quotient in $\Hilb^h(\sh{A})$. Since $\sh{F}$ has codimension $\geq 2$ support, the map $\sh{I}\to\sh{A}$ induces an isomorphism
\[
\sh{I}^{\vee\vee}\to\sh{A}^{\vee\vee}\cong\sh{A}
\]
of $\sh{A}$-modules. This shows that the quotient $\sh{A}\to\sh{F}$ is uniquely determined by its kernel $\sh{I}$, which means that the morphism (\ref{eqn:hilb2M}) is injective.

Next, we show the morphism is \'{e}tale. Let $S$ be a scheme and $S\subset\overline{S}$ a square-zero extension. Fix a flat family of $\sh{A}$-module quotient
\[
0\to\sh{I}\to\sh{A}_S\to\sh{F}\to 0
\]
on $X\times S$. Suppose there exists an $\sh{A}$-module extension $\overline{\sh{I}}$ of $\sh{I}$ over $\overline{S}$, then $\overline{\sh{I}}$ induces a short exact sequence
\[
0\to\overline{\sh{I}}\to \overline{\sh{I}}^{\vee\vee}\to\overline{\sh{F}}\to 0
\]
of $\sh{A}_{\overline{S}}$-modules. Since $\overline{\sh{I}}^{\vee\vee}|_S=\sh{I}^{\vee\vee}\cong\sh{A}_S$, $\overline{\sh{I}}^{\vee\vee}$ is a deformation of $\sh{A}_S$ over $\overline{S}$, which by our assumption, it must be $\sh{A}_{\overline{S}}$. Note that taking double dual $(-)^{\vee\vee}$ in general does not preserve flatness, but it is flat over an open subset. So in our situation, $\sh{I}^{\vee\vee}$ is flat over $S$, which implies it is flat over $\overline{S}$. Thus $\sh{A}_{\overline{S}}\to\overline{F}$ is a deformation (as $\sh{A}$ modules) of the quotient $\sh{A}_S\to\sh{F}$. This is easy to see that the converse is also true: any deformation of the quotient $\sh{A}_S\to\sh{F}$ gives a deformation of $\sh{I}$. The proof is completed.
\end{proof}

Consequently, the Hilbert scheme $\Hilb^h(\sh{A})$ for polynomial $h$ with $\deg(h)\leq 1$ carries a symmetric obstruction theory. We define the DT invariant
\[
\DT^h(\sh{A})=\int_{[\Hilb^h(\sh{A})]_{\vir}}1.
\]
which is equal to the weighted Euler characteristic $\chi\big(\Hilb^h(\sh{A}),\nu_{\Hilb^h(\sh{A})}\big)$.

\begin{example}
Recall that any Azumaya algebra on a Calabi--Yau variety is again Calabi--Yau. Via the correspondence between Azumaya algebras and gerbes, this gives a definition of Donaldson--Thomas invariants on gerbes $\st{X}\to X$ over Calabi--Yau threefolds, which has been studied by \cite{GT13}.
\end{example}

\begin{remark}\label{rmk:dt-ncy}
While we only construct a perfect obstruction theory when $\sh{A}$ is Calabi--Yau. It is expected that (some truncation of) the obstruction theory (\ref{eqn:ob_theory_U}) is perfect for more general $(X,\sh{A})$. For example, if $\sh{A}$ is of global dimension $\leq 2$, then $\Ext_{\sh{A}}^3(\sh{F},\sh{F})=0$ for all $\sh{F}$, thus the obstruction theory is automatically perfect. This gives an analogue of Donaldson invariants for certain non-commutative surfaces.

Another example is if $\sh{A}$ is of global dimension $3$ and $H^i(\sh{A})=0$ for all $i>0$, then it is easy to see that $\Ext^3_{\sh{A}}(\sh{I},\sh{I})=0$ for any ideal sheaf $\sh{I}$ in $\Hilb^h(\sh{A})$ (see \cite[Lemma 2]{MNOP1}). This implies that $\Hilb^h(\sh{A})$ carries a perfect obstruction theory via Proposition~\ref{prop:dt-on-hilb}.
\end{remark}

\begin{example}
We consider the case $X=\Spec(\mathbb{C})$ is a point with $\sh{A}=A$ a finite-dimensional algebra of finite global dimension. Then (stable) coherent $\sh{A}$-modules are exactly finite-dimensional (irreducible) representations of $A$. These DT invariants give virtual counts of irreducible representations. Unfortunately, the moduli space $M^{s,n}(A)$ is in general not projective unless $n=1$.
\end{example}

We end this section by verifying that quantum Fermat quintic threefolds satisfy the assumptions in Proposition~\ref{prop:dt-on-hilb}. 

Let $A$ be a graded algebra associating to a quantum Fermat quintic threefold, and $(X,\sh{A})$ as described in \ref{sec:QFQ}. Since $A$ is a domain, the stalk $\sh{A}_{\eta}$ at the generic point $\eta\in X$ is also a domain. Combining with the fact that $\sh{A}_{\eta}$ is a finite-dimensional algebra over the function field $k(X)$, $\sh{A}_{\eta}$ is a division algebra. Also $X\cong\mathbb{P}^3$ and $\sh{A}$ is locally free, thus $H^1(X,\sh{A})=0$.

\appendix

\bibliographystyle{plain}

\begin{thebibliography}{10}

\bibitem{AZ94}
M.~Artin and J.~J. Zhang.
\newblock Noncommutative projective schemes.
\newblock {\em Adv. Math.}, 109(2):228--287, 1994.

\bibitem{AZ01}
M.~Artin and J.~J. Zhang.
\newblock Abstract {H}ilbert schemes.
\newblock {\em Algebr. Represent. Theory}, 4(4):305--394, 2001.

\bibitem{BF97}
K.~Behrend and B.~Fantechi.
\newblock The intrinsic normal cone.
\newblock {\em Invent. Math.}, 128(1):45--88, 1997.

\bibitem{Beh09}
Kai Behrend.
\newblock Donaldson-{T}homas type invariants via microlocal geometry.
\newblock {\em Ann. of Math. (2)}, 170(3):1307--1338, 2009.

\bibitem{Cal00}
Andrei~Horia Caldararu.
\newblock {\em Derived categories of twisted sheaves on {C}alabi-{Y}au
  manifolds}.
\newblock ProQuest LLC, Ann Arbor, MI, 2000.
\newblock Thesis (Ph.D.)--Cornell University.

\bibitem{CMPS}
Alberto Cazzaniga, Andrew Morrison, Brent Pym, and Bal\'{a}zs Szendr\H{o}i.
\newblock Motivic {D}onaldson-{T}homas invariants of some quantized threefolds.
\newblock {\em J. Noncommut. Geom.}, 11(3):1115--1139, 2017.

\bibitem{GT13}
Amin Gholampour and Hsian-Hua Tseng.
\newblock On {D}onaldson-{T}homas invariants of threefold stacks and gerbes.
\newblock {\em Proc. Amer. Math. Soc.}, 141(1):191--203, 2013.

\bibitem{HL10}
Daniel Huybrechts and Manfred Lehn.
\newblock {\em The geometry of moduli spaces of sheaves}.
\newblock Cambridge Mathematical Library. Cambridge University Press,
  Cambridge, second edition, 2010.

\bibitem{HT10}
Daniel Huybrechts and Richard~P. Thomas.
\newblock Deformation-obstruction theory for complexes via {A}tiyah and
  {K}odaira-{S}pencer classes.
\newblock {\em Math. Ann.}, 346(3):545--569, 2010.

\bibitem{JS12}
Dominic Joyce and Yinan Song.
\newblock A theory of generalized {D}onaldson-{T}homas invariants.
\newblock {\em Mem. Amer. Math. Soc.}, 217(1020):iv+199, 2012.

\bibitem{Kan15}
Atsushi Kanazawa.
\newblock Non-commutative projective {C}alabi-{Y}au schemes.
\newblock {\em J. Pure Appl. Algebra}, 219(7):2771--2780, 2015.

\bibitem{KS08}
Maxim Kontsevich and Yan Soibelman.
\newblock {Stability structures, motivic Donaldson-Thomas invariants and
  cluster transformations}.
\newblock 2008.

\bibitem{Lau79}
Olav~Arnfinn Laudal.
\newblock {\em Formal moduli of algebraic structures}, volume 754 of {\em
  Lecture Notes in Mathematics}.
\newblock Springer, Berlin, 1979.

\bibitem{Low05}
Wendy Lowen.
\newblock Obstruction theory for objects in abelian and derived categories.
\newblock {\em Comm. Algebra}, 33(9):3195--3223, 2005.

\bibitem{MNOP1}
D.~Maulik, N.~Nekrasov, A.~Okounkov, and R.~Pandharipande.
\newblock Gromov-{W}itten theory and {D}onaldson-{T}homas theory. {I}.
\newblock {\em Compos. Math.}, 142(5):1263--1285, 2006.

\bibitem{MR01}
J.~C. McConnell and J.~C. Robson.
\newblock {\em Noncommutative {N}oetherian rings}, volume~30 of {\em Graduate
  Studies in Mathematics}.
\newblock American Mathematical Society, Providence, RI, revised edition, 2001.
\newblock With the cooperation of L. W. Small.

\bibitem{Sim94}
Carlos~T. Simpson.
\newblock Moduli of representations of the fundamental group of a smooth
  projective variety. {I}.
\newblock {\em Inst. Hautes \'{E}tudes Sci. Publ. Math.}, (79):47--129, 1994.

\bibitem{Sze08}
Bal\'{a}zs Szendr\H{o}i.
\newblock Non-commutative {D}onaldson-{T}homas invariants and the conifold.
\newblock {\em Geom. Topol.}, 12(2):1171--1202, 2008.

\bibitem{Tho00}
R.~P. Thomas.
\newblock A holomorphic {C}asson invariant for {C}alabi-{Y}au 3-folds, and
  bundles on {$K3$} fibrations.
\newblock {\em J. Differential Geom.}, 54(2):367--438, 2000.

\bibitem{Zha96}
J.~J. Zhang.
\newblock Twisted graded algebras and equivalences of graded categories.
\newblock {\em Proc. London Math. Soc. (3)}, 72(2):281--311, 1996.

\end{thebibliography}

\end{document}